\documentclass[12pt]{amsart}
\usepackage{geometry}
\usepackage[mathscr]{euscript}
\usepackage{ amssymb }
\usepackage[comma, numbers]{natbib}
\usepackage{bm}
\usepackage{enumerate}
\usepackage[english]{babel}
\usepackage{commath}
\usepackage{tikz-cd}
\usepackage{graphicx}
\usepackage [autostyle, english = american]{csquotes}
\MakeOuterQuote{"}

\theoremstyle{plain}
\newtheorem{theorem}{Theorem}[section]

\theoremstyle{plain}
\newtheorem{lemma}[theorem]{Lemma}

\theoremstyle{plain}
\newtheorem{corollary}[theorem]{Corollary}

\theoremstyle{definition}
\newtheorem{definition}[theorem]{Definition}

\theoremstyle{plain}
\newtheorem{proposition}[theorem]{Proposition}

\theoremstyle{remark}
\newtheorem{remark}[theorem]{Remark}

\theoremstyle{definition}
\newtheorem{example}[theorem]{Example}

\theoremstyle{plain}

\theoremstyle{plain}

\theoremstyle{plain}

\title[Boundaries of coarse proximity spaces]{Boundaries of coarse proximity spaces and boundaries of compactifications}

\author{Pawel Grzegrzolka}
\address{Syracuse University, Syracuse, USA}
\email{pgrzegrz@syr.edu}

\author{Jeremy Siegert}
\address{Ben Gurion University of the Negev, Beer Sheva, Israel} 
\email{siegertj@post.bgu.ac.il}
\date{\today} 


\begin{document}

\begin{abstract}
In this paper, we introduce the boundary $\mathcal{U}X$ of a coarse proximity space $(X,\mathcal{B},{\bf b}).$  This boundary is a subset of the boundary of a certain Smirnov compactification. We show that $\mathcal{U}X$ is compact and Hausdorff and that every compactification of a locally compact Hausdorff space induces a coarse proximity structure whose corresponding boundary is the boundary of the compactification. We then show that many boundaries of well-known compactifications arise as boundaries of coarse proximity spaces. In particular, we give four coarse proximity structures whose boundaries are the Gromov, visual, Higson, and Freudenthal boundaries.
\end{abstract}

\maketitle
\tableofcontents

\section{Introduction}
The field of coarse geometry (occasionally called coarse topology) can be pursued using two different, but by no means mutually exclusive perspectives. The first perspective is one we may call "geometry going to infinity" in which one typically takes notions such as uniformly bounded families and scales as primitives and pursues large scale properties of spaces (metric or otherwise) by considering properties of uniformly bounded families at ever growing "scales." Examples of structures well-suited to this perspective are metric spaces, coarse spaces (see \cite{Roe}), and large scale spaces (see \cite{Dydak}). Coarse properties that exemplify the utility of this perspective are asymptotic dimension originally defined by Gromov in \cite{Gromov} (and expanded to more general structures in \cite{Roe} and \cite{Dydak}), Property A defined by Yu in \cite{Yu}, and amenability (see for example \cite{Roe}). The first two of these properties are particularly well-known for their relationship to the Coarse Baum-Connes and Novikov conjectures (see \cite{Yu}). The other perspective may be called "geometry at infinity." This perspective takes unboundedness and asymptotic disjointness as primitives. The pursuit of large-scale properties of spaces is then conducted by considering how unbounded sets interact with each other "at infinity." Typically, this is done by assigning a topological space to an unbounded space (metric, coarse, or large scale) and considering how the large-scale properties of the base space are reflected in this topological space. The chief example in coarse geometry of such an assignment is the Higson corona (see \cite{Roe}) assigned to a proper metric space. Another such assignment is the Gromov boundary assigned to a hyperbolic metric space (see \cite{Drutu}). Techniques for studying these spaces include the use of asymptotic neighborhoods (see \cite{Dranishnikovasymptotictopology}), coarse neighborhoods (see \cite{Weighill}), asymptotic resemblance spaces (see \cite{Honari}), and coarse proximity spaces (see \cite{paper1}). 

In this paper, we will be exclusively concerned with the "geometry at infinity" perspective. Our primary contribution will be providing a geometrically intuitive "common language" with which one can speak of coarse invariants such as the Higson corona of proper metric spaces and the Gromov boundary of hyperbolic metric spaces. 
Specifically, we will show how to assign a certain boundary space to each coarse proximity space (see Section \ref{discrete extensions and boundaries of coarse proximity spaces}). These boundaries are compact Hausdorff spaces that arise as the boundaries of certain small-scale proximity relations induced by coarse proximities. We will show how basic topological notions of the boundary of a coarse proximity space are captured by the underlying coarse proximity structure (see for example Proposition \ref{coarsely close sets have intersecting trace} and Proposition \ref{equality_at_infinity}). Our secondary contribution is showing that every compactification of a locally compact Hausdorff space induces a coarse proximity structure whose corresponding boundary is the boundary of the compactification (see Theorem \ref{general_theorem}). We use this result to describe coarse proximity structures that give rise to well known boundaries in later sections. As a corollary, we will show that every compact Hausdorff space arises as the boundary of some coarse proximity space (one may compare this result to the results of \cite{yamashitamine}).

The structure of this paper is as follows. In Section \ref{clusters in proximity spaces and the smirnov compactification}, we review the Smirnov compactification of a separated proximity space and associated concepts. Likewise, in Section \ref{coarse proximities} we review the basic definitions and concepts surrounding coarse proximity spaces, as found in \cite{paper1} and \cite{paper2}. Then in Section \ref{discrete extensions and boundaries of coarse proximity spaces}, we introduce the discrete extension of a coarse proximity and use it to define the boundary of a coarse proximity space. After exploring a few basic properties of boundaries of coarse proximity spaces, we show that the assignment of the boundary to a coarse proximity space makes up a functor from coarse proximity spaces to compact Hausdorff spaces. 
In Section \ref{compact hausdorff spaces as boundaries of coarse proximity spaces}, we show how every compactification of a locally compact Hausdorff space induces a coarse proximity structure whose boundary is homeomorphic to the boundary of the compactification. Finally, in Sections \ref{the gromov boundary}, \ref{The Visual Boundary}, \ref{the higson boundary}, and \ref{the freudenthal boundary}  we describe the coarse proximity structures on proper hyperbolic metric, complete proper Cat(0), proper metric, and locally compact Hausdorff spaces whose boundaries are the the Gromov boundary, the visual boundary, the Higson corona, and the Freudenthal boundary, respectively.

\vspace{\baselineskip}

Throughout the paper, we will use the following notation and conventions:
\begin{itemize}
\item If $X$ is a topological space and $A$ is a subset of $X$, the closure of $A$ in $X$ will be denoted by $cl_X(A)$. When there is no risk of ambiguity, the subscript $X$ will be dropped, and the closure of $A$ in $X$ will be denoted by $cl(A).$
\item A compactification of a topological space $X$ will be denoted by $\overline{X}$, with the exception of the Smirnov compactification, which will be denoted by $\mathfrak{X}$.
\item If $X$ is a topological space and $\overline{X}$ is its compactification, then by the trace of $A\subseteq X$ in $B\subseteq \overline{X}$ we will mean the intersection of the closure of $A$ in $\overline{X}$ with $B$. We will denote this trace by $tr_{\overline{X}, B}(A)$. In other words,
\[tr_{\overline{X}, B}(A)= cl_{\overline{X}}(A) \cap B.\]
\item By a proper metric space we mean a metric space $(X,d)$ whose closed and bounded sets are compact.
\end{itemize}
\vspace{1em}

The authors are grateful to the reviewers whose helpful comments and suggestions greatly improved the quality of this paper.

\section{Proximities and the Smirnov Compactification}\label{clusters in proximity spaces and the smirnov compactification}
 
In this section, we recall necessary definitions and theorems related to proximities from \cite{proximityspaces}. In particular, we discuss clusters and the Smirnov Compactification. For an introduction to the subject, see \cite{proximityspaces}.

\begin{definition}
Let $X$ be a set. A binary relation $\delta$ (or $\delta_X$ if the base space is not clear from the context) on the power set of $X$ is called a {\bf proximity} on $X$ if it satisfies the following axioms for all $A,B,C \subseteq X$:
\begin{enumerate}
	\item $A\delta B\implies B\delta A,$
	\item $A\delta B\implies A,B\neq\emptyset,$
	\item $A \cap B\neq\emptyset\implies A\delta B,$
	\item $A\delta (B\cup C)\iff A\delta B\text{ or }A\delta C,$ \label{axiom44}
	\item $A\bar{\delta}B\implies\exists E\subseteq X,\,A\bar{\delta}E\text{ and }(X\setminus E)\bar{\delta}B,$ \label{axiom55}
\end{enumerate}
where $A\bar{\delta}B$ means $``A\delta B"$ does not hold. Axiom (\ref{axiom55}) is called the \textbf{strong axiom}. A pair $(X,\delta)$ where $X$ is a set and $\delta$ is a proximity on $X$ is called a {\bf proximity space}.
\end{definition}

\begin{definition}
A proximity space $(X, \delta)$ is called \textbf{separated} if 
\[\{x\}\delta\{y\} \Longleftrightarrow x=y\]
for all $x,y \in X.$
\end{definition}

\begin{example}
Let $(X,d)$ be a metric space. Then the relation $\delta_d$ 
 defined by
\[A\delta_d B \Longleftrightarrow d(A,B)=0\]
is a separated proximity called the \textbf{metric proximity}, where 
\[d(A,B)=\inf\{d(a,b) \mid a \in A, b \in B\}.\]
\end{example}

\begin{example}\label{subspace_proximity}
Let $(X, \delta_X)$ be a proximity space and let $Y\subseteq X.$ Then the relation $\delta_Y$ defined by
\[A\delta_Y B \iff A \delta_X B,\]
where $A$ and $B$ are subsets of $Y$, is a proximity relation on $Y,$ called the \textbf{subspace proximity}.
\end{example}

\begin{definition}
Let $(X, \delta_X)$ and $(Y, \delta_Y)$ be two proximity spaces. A function $f:X \to Y$ is called a \textbf{proximity map} if
\[A \delta_X B \implies f(A) \delta_Y f(B).\]
A bijective proximity map whose inverse is also a proximity map is called a \textbf{proximity isomorphism}. A proximity isomorphism onto a subspace of a proximity space in called a \textbf{proximity embedding}.
\end{definition}

\begin{theorem}\label{closure}
Let $(X,\delta)$ be a proximity space. Then $\delta$ induces a topology on $X$ defined by
\[A \text{ is closed} \quad \Longleftrightarrow \quad (A\delta \{x\} \implies x \in A).\]
This topology is always completely regular, and is Hausdorff if and only if $\delta$ is separated.
\end{theorem}
\begin{proof}
See \cite{proximityspaces}.
\end{proof}

It is easy to show that the metric proximity induces the metric topology, and that the subspace proximity induces the subspace topology. From now on, we will always assume that a given proximity space is equipped with the induced topology, as in Theorem \ref{closure}.

\begin{example}\label{unique_proximity}
Let $X$ be a compact Hausdorff space. Then the relation $\delta$ defined by
\[A\delta B \iff cl(A)\cap cl(B) \neq \emptyset\]
for any $A,B \subseteq X$ is the unique separated proximity on $X$ inducing the original topology on $X$.
\end{example}

\begin{proposition}\label{prop_about_closures}
Let $(X, \delta)$ be a proximity space, and let $A,B \subseteq X.$ Then
\begin{enumerate}
\item $cl(A)=\{x\in X \mid \{x\} \delta A\},$
\item $A \delta B \Longleftrightarrow (cl(A)) \delta (cl(B)).$
\end{enumerate}
\end{proposition}
\begin{proof}
See \cite{proximityspaces}.
\end{proof}

We now introduce clusters, which are the building blocks of the Smirnov Compactification. 

\begin{definition}\label{clusters}
	Let $(X,\delta)$ be a proximity space. Let $\sigma$ be a collection of subsets of $X$. Then $\sigma$ is called a {\bf cluster} if the following hold:
	\begin{enumerate}
		\item if $A,B\in\sigma,$ then $A\delta B,$
		\item if $A\delta B$ for all $B\in\sigma,$ then $A\in\sigma,$
		\item if $(A\cup B)\in\sigma,$ then either $A\in\sigma$ or $B\in\sigma$.
	\end{enumerate}
\end{definition}

\begin{definition}
A cluster in a proximity space $(X,\delta)$ is called a {\bf point cluster} if $\{x\}\in\sigma$ for some $x\in X$. If $(X,\delta)$ is separated, we denote the point cluster containing $x$ by $\sigma_{x},$ as each cluster in a separated proximity space contains at most one singleton because distinct points are not close in separated proximity spaces.
\end{definition}

A simple result about clusters that will be useful to us in Section \ref{the freudenthal boundary} is the following:

\begin{proposition}\label{cluster_containing_both}
	Let $(X,\delta)$ be a separated proximity space and $A,B\subseteq X$ such that $A\delta B$. Then there is a cluster $\sigma$ in $X$ that contains both $A$ and $B$.
\end{proposition}
\begin{proof}
See 
\cite{proximityspaces}.
\end{proof}

We now focus on the construction of the Smirnov Compactification.

\begin{definition}
	Let $(X,\delta)$ be a proximity space. Define $\mathfrak{X}$ to be the set of all clusters in $X.$ Let $\mathcal{A}\subseteq\mathfrak{X}$. We say that a subset $A\subseteq X$ {\bf absorbs} $\mathcal{A}$ if $A\in\sigma$ for all $\sigma\in\mathcal{A}$.
\end{definition}

\begin{theorem}
	Let $(X,\delta)$ be a separated proximity space and $\mathfrak{X}$ the corresponding set of all clusters in $X.$ For two subsets $\mathcal{A},\mathcal{B}\subseteq\mathfrak{X},$ define:
\[\mathcal{A}\delta^{*}\mathcal{B} \quad \Longleftrightarrow \quad A\delta B\]
for all  $A,B\subseteq X$ that absorb $\mathcal{A}$ and $\mathcal{B}$, respectively. The relation $\delta^{*},$ called the \textbf{Smirnov proximity}, is a separated proximity on $\mathfrak{X}$ that induces a compact Hausdorff topology on $\mathfrak{X}.$ The mapping $f:X\rightarrow\mathfrak{X}$ defined by $f(x)=\sigma_{x}$ is a dense proximity embedding. The space $(\mathfrak{X}, \delta^{*})$ is called the \textbf{Smirnov compactification} of $(X, \delta).$ It is the unique (up to proximity isomorphism) compact Hausdorff space into which $X$  embeds as a dense subspace through a proximity embedding.
\end{theorem}
\begin{proof}
See Section $7$ in \cite{proximityspaces}.
\end{proof}

Given a proximity space $(X, \delta)$, we will always denote its Smirnov compactification by $\mathfrak{X}$ and the corresponding proximity by $\delta^*$. Also, it is customary to identify $X$ with its image in $\mathfrak{X}.$ We are going to follow this practice. Consequently, given $A \subseteq X,$ one can think of $A$ as a subset of $\mathfrak{X}.$ It is then easy to show that
\[cl_{\mathfrak{X}}(A) = \{ \sigma \in \mathfrak{X} \mid A \in \sigma\}.\]
 In other words, 
\[A \in \sigma \iff \sigma \in cl_{\mathfrak{X}}(A).\]

The following proposition is used to show that every proximity map extends uniquely to a proximity map between the Smirnov compactifications.

\begin{proposition}\label{associated cluster}
	Let $(X, \delta_X)$ and $(Y, \delta_Y)$ be proximity spaces. Let $f:X\rightarrow Y$ be a proximity map. Then to each cluster $\sigma_1$ in $X,$ there corresponds a cluster $\sigma_2$ in $Y$ defined by
	\[\sigma_2=\{A\subseteq Y\mid A\delta_Y f(B)\text{ for all }B\in\sigma_1\}.\]
\end{proposition}
\begin{proof}
See 
 \cite{proximityspaces}.
\end{proof}

Notice that in the setting of the above proposition, it is immediate that the image of any set in $\sigma_1$ is in $\sigma_2$ i.e., $f(\sigma_1) \subseteq \sigma_2.$

\begin{theorem}\label{extension_theorem}
	Let $(X, \delta_X)$ and  $(Y, \delta_Y)$ be separated proximity spaces, and let $(\mathfrak{X}, \delta_X^{*})$ and $(\mathfrak{Y}, \delta_Y^{*})$ be the respective Smirnov compactifications. Let $f: X \to Y$ be a proximity map. Then $f$ extends to a unique proximity map $f^*:\mathfrak{X} \to \mathfrak{Y},$ which is defined using Proposition \ref{associated cluster}.
\end{theorem}
\begin{proof}
The proof of this theorem (with the unnecessary and unused assumption of the surjectivity of $f$) can be found in \cite{proximityspaces}.
\end{proof}

\begin{definition}
	Given a separated proximity space $(X, \delta)$ and its corresponding Smirnov compactification $(\mathfrak{X}, \delta^{*})$, the {\bf Smirnov boundary} of $X$ is the subspace $\mathfrak{X}\setminus X$ (with the subspace proximity inherited from $(\mathfrak{X},\delta^*)$). 
\end{definition}

Now we investigate a few basic properties of the Smirnov boundary.

\begin{proposition}\label{elements of Smirnov boundary don't contain compact sets}
	Let $(X,\delta)$ be a separated proximity space. 
	Let $\sigma$ be an element of $\mathfrak{X}\setminus X$. Then $ K \notin \sigma$ for any compact subset $K$ of $X$.
\end{proposition}
\begin{proof}
	Let $\sigma\in\mathfrak{X}\setminus X$ be given. Assume towards a contradiction that $K\subseteq X$ is a compact set such that $K\in\sigma$.  Identify $K$ with the corresponding set of point clusters. If $B\subseteq X$ absorbs $K$ and $k \in K,$ then $\{k \}\delta B$ (since $\{k \}$ is an element of $\sigma_k$ and $B$ absorbs $\sigma_k$). Consequently, $K \subseteq cl_X(B).$ Also, given any $A\in\sigma$, we must have that $A\delta K$ by the definition of a cluster. Since $K \subseteq cl_X(B)$ for any absorbing set $B$ of $K$, we must have that $A\delta (cl_X(B)).$ Consequently, by Proposition \ref{prop_about_closures}, we have that $A \delta B$ for any $A \in \sigma$ and any set $B$ absorbing $K.$ This yields that $\{\sigma\}\delta^{*}K$ in $\mathfrak{X}$. However, because $K$ is compact in $X$, $K$ is compact in $\mathfrak{X}$, and since compact subsets of Hausdorff spaces are closed, $K$ is closed in $\mathfrak{X}$. Thus, we must have that $\sigma\in K$, which is to say that $\sigma$ is a point cluster based at some point of $K$. This contradicts $\sigma$ being an element of $\mathfrak{X}\setminus X$. 
\end{proof}

\begin{proposition}\label{can subtract by compact sets in cluster}
	Let $(X,\delta)$ be a separated proximity space 
	and $\sigma\in\mathfrak{X}\setminus X$. Then for every $A\in\sigma$ and every compact $K\subseteq X,$ we have that $A\setminus K\in\sigma$.
\end{proposition}
\begin{proof}
This follows immediately from axiom $(3)$ in the definition of a cluster and Proposition \ref{elements of Smirnov boundary don't contain compact sets}.
\end{proof}
 
\section{Coarse Proximities}\label{coarse proximities}
In this section, we recall the needed definitions and theorems from \cite{paper1} and \cite{paper2}. To provide the reader with the geometric intuition behind the definitions, most definitions are followed with an example in a metric setting.

\begin{definition}
	Let $X$ be a set. A {\bf bornology} $\mathcal{B}$ (or $\mathcal{B}_X$ if the base space is not clear from the context) is a family of subsets of $X$ satisfying:
	\begin{enumerate}
		\item $\{x\}\in\mathcal{B}$ for all $x\in X,$
		\item $A\in\mathcal{B}$ and $B\subseteq A$ implies $B\in\mathcal{B},$
		\item if $A,B\in\mathcal{B},$ then $A\cup B\in\mathcal{B}.$
	\end{enumerate}
Elements of $\mathcal{B}$ are called {\bf bounded}, and subsets of $X$ not in $\mathcal{B}$ are called {\bf unbounded}. 
\end{definition}

\begin{example}
Let $(X,d)$ be a metric space. Then the collection of all metrically bounded sets is a bornology. We denote this bornology by $\mathcal{B}_d$.
\end{example}

\begin{example}
Let $X$ be a topological space. Then the collection of all precompact sets (i.e., sets whose closure is compact) is a bornology. We denote this bornology by $\mathcal{B}_C$. \end{example}

\begin{definition}\label{coarseproximitydefinition}
Let $X$ be a set equipped with a bornology $\mathcal{B}$. A \textbf{coarse proximity} on $X$ (or $(X, \mathcal{B})$ if the bornology is not clear from the context) is a relation ${\bf b}$ (or ${\bf b}_X$ if the base space is not clear from the context) on the power set of $X$ satisfying the following axioms for all $A,B,C \subseteq X:$

\begin{enumerate}
	\item $A{\bf b}B \implies B{\bf b}A,$ \label{axiom1}
	\item $A{\bf b}B \implies A \notin \mathcal{B}$ and $B \notin \mathcal{B},$ \label{axiom2}
	\item $A\cap B \notin \mathcal{B} \implies A {\bf b} B,$ \label{axiom3}
	\item $(A \cup B){\bf b}C \Longleftrightarrow A{\bf b}C$ or $B{\bf b}C,$ \label{axiom4}
	\item $A\bar{\bf b}B \implies \exists E \subseteq X$ such that $A\bar{\bf b}E$ and $(X\setminus E)\bar{\bf b}B,$ \label{axiom5}
\end{enumerate}
where $A\bar{ {\bf b}}B$ means "$A{\bf b} B$ is not true." If $A {\bf b} B$, then we say that $A$ is \textbf{coarsely close} to (or \textbf{coarsely near}) $B.$ Axiom (\ref{axiom5}) is called the \textbf{strong axiom}. A triple $(X,\mathcal{B},{\bf b})$ where $X$ is a set, $\mathcal{B}$ is a bornology on $X$, and ${\bf b}$ is a coarse proximity on $X,$ is called a {\bf coarse proximity space}.
\end{definition}

Even though the strong axiom was already defined for proximity spaces, the meaning behind the strong axiom will always be clear from the context.

\begin{example}
Let $(X,d)$ be a metric space with the bornology $\mathcal{B}_d.$ For any $A,B \subseteq X$, define $A{\bf b}_d B$ if and only if there exists $\epsilon < \infty$ such that for all bounded sets $D$, there exists $a \in A \setminus D$ and $b \in B \setminus D$ such that $d(a,b) < \epsilon.$  Then ${\bf b}_d$ is a coarse proximity, called the \textbf{metric coarse proximity}. The triple $(X, \mathcal{B}_d, {\bf b}_d)$ is called the \textbf{metric coarse proximity space.}
\end{example}

To define maps between coarse proximity spaces, we need the concept of weak asymptotic resemblance. 

\begin{theorem}
	Let $(X,\mathcal{B},{\bf b})$ be a coarse proximity space. Define $\phi$ (or $\phi_{\bf  b}$ if the coarse proximity is not clear from the context) to be the relation on the power set of $X$ by $A\phi B$ if and only if the following hold:
	\begin{enumerate}
	\item for every unbounded $B^{\prime}\subseteq B$ we have $A{\bf b}B^{\prime},$
	\item for every unbounded $A^{\prime}\subseteq A$ we have $A^{\prime}{\bf b}B.$
	\end{enumerate}	
	Then  $\phi$ is an equivalence relation satisfying 
	\[A\phi B \text{ and } C \phi D \implies (A\cup C) \phi (B \cup D)\]
	for any $A,B,C,D \subseteq X.$ We call this equivalence relation the \textbf{weak asymptotic resemblance} induced by the coarse proximity ${\bf b}$.
\end{theorem}
\begin{proof}
See \cite{paper1}.
\end{proof}

The above definition of the weak asymptotic resemblance is a specification of a more general notion of weak asymptotic resemblance that can be found in literature (see for example Definition 6.5 in \cite{paper1}). 

We will later show that the relation ${\bf b}$ captures "closeness at infinity" (see Proposition \ref{coarsely close sets have intersecting trace}), and the relation $\phi$ captures "equality at infinity" (see Proposition \ref{equality_at_infinity}).

\begin{example}
Let $(X, \mathcal{B}_d, {\bf b}_d)$ be a metric coarse proximity space. Then the weak asymptotic resemblance induced by ${\bf b}_d$ is denoted by $\phi_d$. It can be shown (see \cite{paper1}) that for nonempty $A,B \subseteq X$,
\[A\phi_d B \Longleftrightarrow \text{$A$ and $B$ have finite Hausdorff distance}.\]
\end{example}

\begin{definition}
	Let $(X,\mathcal{B}_X,{\bf b}_X)$ and $(Y,\mathcal{B}_Y,{\bf b}_Y)$ be coarse proximity spaces. Let $f:X\rightarrow Y$ be a function. Then $f$ is a \textbf{coarse proximity map} provided that the following are satisfied for all $A,B \subseteq X$:
\begin{enumerate}
\item $B\in\mathcal{B}_X \implies f(B)\in\mathcal{B}_Y,$ \label{item11}
\item $A{\bf b}_XB \implies f(A){\bf b}_Yf(B).$ \label{item22}
\end{enumerate}
\end{definition}
Notice that item \ref{item22} from the above definition implies that item \ref{item11} can be strengthened to 
\[B\in\mathcal{B}_X \Longleftrightarrow f(B)\in\mathcal{B}_Y.\]
\begin{definition}\label{coarsecloseness}
	Let $X$ be a set and $(Y,\mathcal{B},{\bf b})$ a coarse proximity space. Two functions $f,g:X \to Y$ are {\bf close} if for all $A \subseteq X$
	\[f(A)\phi_{\bf b} g(A).\]
\end{definition}

\begin{definition}
Let $(X,\mathcal{B}_X,{\bf b}_X)$ and $(Y,\mathcal{B}_Y,{\bf b}_Y)$ be coarse proximity spaces. We call a coarse proximity map $f: X \to Y$ a \textbf{coarse proximity isomorphism} if there exists a coarse proximity map $g:Y\to X$ such that $g\circ f$ is close to $id_{X}$ and $f\circ g$ is close to $id_{Y}.$ We say that $(X,\mathcal{B}_X,{\bf b}_X)$ and $(Y,\mathcal{B}_Y,{\bf b}_Y)$ are \textbf{coarse proximity isomorphic} (or just \textbf{isomorphic}) if there exists a coarse proximity isomorphism $f:X \to Y.$
\end{definition}

The collection of coarse proximity spaces and closeness classes of coarse proximity maps makes up the category {\bf CrsProx} of coarse proximity spaces. For details, see \cite{paper1}.

\section{Discrete Extensions and Boundaries of Coarse Proximity Spaces}\label{discrete extensions and boundaries of coarse proximity spaces}

In this section, we introduce discrete extensions of coarse proximities. We use them to define the boundaries of coarse proximity spaces. We also prove a few properties of such boundaries. In particular, we show that every such boundary is compact and Hausdorff. Finally, we show the existence of a nontrivial functor from the category of coarse proximity spaces (with closeness classes of coarse proximity maps) to the category of compact separated proximity spaces (with proximity maps). To learn more about the category of coarse proximity spaces, the reader is referred to Section 7 in \cite{paper1}.

The following definition has been inspired by the definition of the Higson proximity of a proper metric space (see Definition \ref{higson proximity}).

\begin{proposition}\label{induced proximity}
	Let $(X,\mathcal{B},{\bf b})$ be a coarse proximity space. Define a binary relation $\delta_{dis}$ (or $\delta_{dis,{\bf b}}$ if the coarse proximity is not clear from the context) on the power set of $X$ by
	\[A\delta_{dis} B \quad \Longleftrightarrow \quad A\cap B\neq\emptyset  \text{ or } A{\bf b}B.\]
Then $\delta_{dis}$ is a separated proximity on $X.$ We will call this proximity the {\bf discrete extension} of ${\bf b}$ and the space $(X, \delta_{dis})$ the \textbf{discrete extension} of $(X,\mathcal{B},{\bf b})$.
\end{proposition}

\begin{proof}
All the axioms besides the strong axiom are immediate. To prove the strong axiom, let $A \bar{\delta}_{dis} B.$ Then $A \cap B = \emptyset$ and $A \bar{\bf b} B.$ Consequently, there exists $E \subseteq X$ such that $A \bar{\bf b} E$ and $(X \setminus E) \bar{\bf b} B.$ Since $A \cap E$ is bounded, the set $E'= E \setminus A$ is disjoint from $A$, and we still have $A \bar{\bf b} E'$ and $(X \setminus E') \bar{\bf b} B.$ Since $B$ is contained in $E'$ up to a bounded set, $B \setminus E'$ is bounded. Consequently, $E''=E' \cup (B \setminus E')$ is still disjoint from $A$ (we added a subset of $B$, which does not intersect $A$), and we still have that $A \bar{\bf b} E''$ and $(X \setminus E'') \bar{\bf b} B.$ By construction, $E''$ fully contains $B,$ i.e., $X \setminus E''$ does not intersect $B.$ In conclusion, we found $E''$ such that $A \cap E'' = \emptyset,$ $A \bar{\bf b} E'',$ $(X \setminus E'') \cap B = \emptyset,$ and $(X \setminus E'') \bar{\bf b} B.$ This means that $A \bar{\delta}_{dis} E''$ and $(X \setminus E'') \bar{\delta}_{dis} B,$ which completes the proof of the strong axiom.
\end{proof}

Now we are ready to define boundaries of coarse proximity spaces.
\begin{definition}
Let $(X,\mathcal{B},{\bf b})$ be a coarse proximity space. Let $\mathfrak{X}$ be the Smirnov compactification induced by $\delta_{dis}$.  Define $\mathcal{U}X\subseteq\mathfrak{X}$ (or $\mathcal{U}_{\bf b}X$ if the coarse proximity inducing the discrete extension is not clear from the context) to be the set of clusters in $(X,\delta_{dis})$ that do not contain any bounded sets. In particular, $\mathcal{U}X \subseteq \mathfrak{X} \setminus X.$ The space $\mathcal{U}X$ (equipped with the subspace proximity inherited from $\mathfrak{X}$) is called the \textbf{boundary of the coarse proximity space} $(X,\mathcal{B},{\bf b}).$
\end{definition}

Given a coarse proximity space $(X,\mathcal{B},{\bf b})$, we will always assume that $(\mathfrak{X}, \delta_{dis}^*)$ (or just $\mathfrak{X}$) denotes the Smirnov compactification induced by $\delta_{dis}$. To show that the boundary of a coarse proximity space is compact and Hausdorff, we need the following lemmas.

\begin{lemma}\label{clusters realized by coarse proximity}
Let $(X,\mathcal{B},{\bf b})$ be a coarse proximity space. Let $\sigma$ be an arbitrary element of $\mathcal{U}X.$ Then for any $A,B \in \sigma,$ we have $A{\bf b}B.$
\end{lemma}
\begin{proof}
Since $A$ and $B$ are in $\sigma,$ we know that $A \delta_{dis} B.$ Then $A \cap B \neq \emptyset$ or $A{\bf b} B.$ If $A \cap B$ is unbounded, then clearly $A {\bf b} B.$ Suppose that $A \cap B$ is bounded. By axiom 3 of a cluster, we know that $A\setminus (A \cap B)$ is in $\sigma$ or $A \cap B$ is in $\sigma.$ However, $A \cap B$ is bounded, so it cannot belong to $\sigma.$ Thus, it has to be that  $A\setminus (A \cap B)$ is in $\sigma.$ But then $(A\setminus (A \cap B)) \delta_{dis} B$ (since $A\setminus (A \cap B)$ is in $\sigma$ and $B$ is in $\sigma$). Since these sets are disjoint, this means that $(A\setminus (A \cap B)) {\bf b} B.$ Since $A \cap B$ is bounded, this in turn implies that $A{ \bf b}B.$
\end{proof}

\begin{lemma}\label{phi_relation_closed}
Let $(X,\mathcal{B},{\bf b})$ be a coarse proximity space. Then every element of $\mathcal{U}X$ is closed under the weak asymptotic resemblance $\phi$ induced by ${\bf b}$. In other words, if $A \in \sigma \in \mathcal{U}X,$ then for every $B$ such that $A \phi B,$ we have $B \in \sigma.$
\end{lemma}
\begin{proof}
Let $\sigma$ be an element of $\mathcal{U}X,$ $A$ an element of $\sigma,$ and $B$ a subset of $X$ such that $A \phi B.$ We need to show that $B \in \sigma.$ Notice that since $A \phi B$ and $A$ is unbounded, we know that $B$ is unbounded. We are going to utilize axiom 2 of a cluster. Let $C$ be an arbitrary element of $\sigma.$ We want to show that $B{\bf b} C,$ and consequently $B \delta_{dis} C.$ We know that $A \delta_{dis} C$ (because $A$ and $C$ belong to the same cluster $\sigma$). By the previous lemma, this means that $A{\bf b}C.$ Since $A$ and $B$ are $\phi$ related, they are coarsely close to the same subsets of $X$ (see Corollary 6.18 in \cite{paper1}). Thus, it has to be true that $B{\bf b}C,$ and consequently, $B \delta_{dis} C.$ Thus, we have shown that for an arbitrary element $C$ of $\sigma,$ we have $B \delta_{dis} C.$ By axiom 2 of a cluster, this shows that $B$ is in $\sigma.$
\end{proof}

\begin{proposition}
	For every coarse proximity space $(X,\mathcal{B},{\bf b}),$ the space $\mathcal{U}X$ is compact and Hausdorff.
\end{proposition}
\begin{proof}
	First note that if $X$ is bounded, then $\mathcal{U}X$ is empty. Thus, assume that $X$ is unbounded. Because $\mathfrak{X}$ is a compact Hausdorff space and $\mathcal{U}X\subseteq\mathfrak{X}$ it will suffice to show that $\mathcal{U}X$ is closed in $\mathfrak{X}$. By Proposition \ref{prop_about_closures}, it suffices then to show that if $\sigma\in\mathfrak{X}\setminus\mathcal{U}X,$ then $\{\sigma\}$ is not close to $\mathcal{U}X$ in $\mathfrak{X}$. Suppose $\sigma\in\mathfrak{X}\setminus\mathcal{U}X.$ This means that there is some bounded set $B\subseteq X$ such that $B\in\sigma$. Let $\phi$ be the weak asymptotic resemblance induced by ${\bf b}$. Since $B$ is bounded, $(X\setminus B)\phi X.$ By Lemma \ref{phi_relation_closed}, every element of $\mathcal{U}X$ is closed under the $\phi$ relation. Because $X$ is an element of each cluster of $X$ we have that the set $X\setminus B$ belongs to each element of $\mathcal{U}X$. Because $B\bar{\delta}_{dis}(X\setminus B)$, $B$ absorbs $\{\sigma\}$, and $X\setminus B$ absorbs $\mathcal{U}X$ we have that $\{\sigma\}$ is not close to $\mathcal{U}X$ in the Smirnov compactification. Therefore, $\mathcal{U}X$ is closed in $\mathfrak{X},$ and is consequently compact and Hausdorff.
\end{proof}

For the reminder of this section, for any $A \subseteq X$, $tr(A)$ will denote $tr_{\mathfrak{X}, \mathcal{U}X}(A),$ as defined in the introduction.

What follows are some basic facts regarding the boundary of a coarse proximity space.

\begin{proposition}\label{nonempty trace}
	Let $(X,\mathcal{B},{\bf b})$ be a coarse proximity space, and $A\subseteq X.$ Then 
\begin{enumerate}
\item $tr(A) = \emptyset$ if $A$ is bounded, \label{item1}
\item $tr(A) \neq \emptyset$ if $A$ is unbounded. \label{item2}
\end{enumerate}
\end{proposition}
\begin{proof}
To see (\ref{item1}), let $A$ be bounded. Then $A \in \sigma$ for any $\sigma \in cl_{\mathfrak{X}}(A)$ (it is because $cl_{\mathfrak{X}}(A)=\{ \sigma \in \mathfrak{X} \mid A \in \sigma\}$). Consequently, $tr(A)=cl_{\mathfrak{X}}(A)\cap\mathcal{U}X= \emptyset.$ To see (\ref{item2}), notice that for each $D\in\mathcal{B},$ we have that $A\setminus D$ is unbounded. The collection $\mathcal{A}=\{cl_{\mathfrak{X}}(A\setminus D)\mid D\in\mathcal{B}\}$ is the collection of nonempty closed sets in $\mathfrak{X}$ that trivially has the finite intersection property, and thus $\bigcap\mathcal{A}\neq\emptyset$. Let $\sigma$ be a point in this intersection. Clearly $\sigma \in cl_{\mathfrak{X}}(A)=cl_{\mathfrak{X}}(A\setminus \emptyset).$ Also, $\sigma$ must be an element of $\mathcal{U}X.$ To see that, for contradiction assume that there exists $C\in\mathcal{B}$ such that $C\in\sigma.$ Since $\sigma\in cl_{\mathfrak{X}}(A\setminus D)$ for each $D\in\mathcal{B},$ we must have that $A\setminus D \in \sigma$ for each $D\in\mathcal{B}$. Consequently, $C\delta_{dis}(A\setminus D)$ for each $D\in\mathcal{B}.$ 
In particular, $C\delta_{dis}(A\setminus C).$ However, this is a contradiction, since $C$ is bounded and $A\setminus C$ is unbounded (and thus $C \bar{\bf b} (A\setminus C)$) and $C\cap(A\setminus C)=\emptyset.$ Therefore, $tr(A)\neq\emptyset$.
\end{proof}

The following proposition explains the intuitive notion of coarse proximities capturing "closeness at infinity." In particular, it shows that two subsets of a coarse proximity space are coarsely close if and only if they are "close at the boundary."

\begin{proposition}\label{coarsely close sets have intersecting trace}
	Let $(X,\mathcal{B},{\bf b})$ be a coarse proximity space. Let $A$ and $B$ be subsets of $X.$ Then $A{\bf b}B$ if and only if $tr(A)\cap tr(B)\neq\emptyset$.
\end{proposition}
\begin{proof}
The proposition is clearly true if at least one of the subsets is bounded, so let us assume that both $A$ and $B$ are unbounded. To prove the forward direction, assume $A{\bf b}B.$ Consequently, we have that $(A\setminus D_{1}){\bf b}(B\setminus D_{2})$ for all $D_{1},D_{2}\in\mathcal{B}$. This implies that $(A\setminus D_{1})\delta_{dis}(B\setminus D_{2})$ for all $D_{1},D_{2}\in\mathcal{B}$. Because $\mathfrak{X}$ is compact we have that $cl_{\mathfrak{X}}(A\setminus D_{1})\cap cl_{\mathfrak{X}}B\neq\emptyset$ for all $D_{1}\in\mathcal{B}.$ We then have that the collection $\mathcal{C}=\{cl_{\mathfrak{X}}(B)\}\cup\{cl_{\mathfrak{X}}(A\setminus D)\mid D\in\mathcal{B}\}$ is a collection of closed sets in $\mathfrak{X}$ that has the finite intersection property. To see this, notice that for any bounded sets $D_{1},\ldots, D_{n}\in\mathcal{B},$ we have
\[\left(\bigcap_{i=1}^{n}cl_{\mathfrak{X}}(A\setminus D_{i})\right)\cap cl_{\mathfrak{X}}(B) \supseteq cl_{\mathfrak{X}}\left(A\setminus\left(\bigcup_{i=1}^{n}D_{i}\right)\right)\cap cl_{\mathfrak{X}}(B)\neq\emptyset.\]	
The compactness of $\mathfrak{X}$ then tells us that $\bigcap\mathcal{C}\neq\emptyset.$ Let $\sigma \in \bigcap\mathcal{C}.$ The proof of Proposition \ref{nonempty trace} gives us that $\sigma\in\mathcal{U}X.$ Moreover, it is clear that $\sigma \in cl_{\mathfrak{X}}(A)$ and $\sigma \in  cl_{\mathfrak{X}}(B).$ Thus, $\sigma\in tr(A)$ and $\sigma\in tr(B),$ which shows that $tr(A)\cap tr(B) \neq\emptyset.$

To see the converse, assume that $tr(A)\cap tr(B)\neq\emptyset$. Let $\sigma \in tr(A)\cap tr(B).$ Then $A,B\in\sigma$ and $\sigma \in \mathcal{U}X,$ and thus Lemma \ref{clusters realized by coarse proximity} gives us that $A{\bf b}B.$
\end{proof}

The following proposition shows that two subsets of a coarse proximity space are $\phi$ related if and only if they are "the same at the boundary."

\begin{proposition}\label{equality_at_infinity}
	Let $(X,\mathcal{B},{\bf b})$ be a coarse proximity space and $\phi$ the corresponding weak asymptotic resemblance. If $A,B\subseteq X,$ then $A\phi B$ if and only if $tr(A)=tr(B).$
\end{proposition}
\begin{proof}
	If $A$ and $B$ are bounded, then clearly $A\phi B$ and $tr(A)=tr(B)=\emptyset$. If only one of the sets is bounded, then clearly $A \bar{\phi} B$ and $tr(A) \neq tr(B).$ So assume that $A$ and $B$ are unbounded. To prove the forward direction, assume $A\phi B.$ Let $\sigma \in tr(A).$ This means that $A \in \sigma$ and $\sigma \in \mathcal{U}X.$ Then Lemma \ref{phi_relation_closed} implies that $B\in \sigma,$ i.e., $\sigma \in cl_{\mathfrak{X}}(B).$ Since $\sigma \in \mathcal{U}X$ as well, this shows that $\sigma \in tr(B).$ Thus, $tr(A) \subseteq tr(B)$ and by symmetry of the argument it follows that $tr(A)=tr(B).$ To prove the converse, assume $tr(A)=tr(B).$ Let $C\subseteq A$ be an unbounded subset. For contradiction, assume that $C\bar{ {\bf b}}B.$ Then $C\cap B\in\mathcal{B}.$ We then have that $C_{0}=(C\setminus(C\cap B))$ is an unbounded subset of $A$ such that $C_{0}\bar{\bf b}B$. Because $C_{0}\subseteq A$ we trivially have that $tr(C_{0})\subseteq tr(A)=tr(B)$. Proposition \ref{nonempty trace} tells us that $tr(C_{0})\neq\emptyset$. Thus, $tr(C_{0})\cap tr(B)\neq\emptyset$. Proposition \ref{coarsely close sets have intersecting trace} implies then that $C_{0}{\bf b}B,$ a contradiction. Thus, $C{\bf b}B$. One can similarly show that if $C\subseteq B$ is an unbounded set, then $C{\bf b}A$. Therefore, $A\phi B$. 
\end{proof}

For the remainder of this section, we will focus on the construction of the aforementioned functor (from the category of coarse proximity spaces to the category of compact Hausdorff spaces).

\begin{proposition}\label{coarse_proximity_map_induces_proximity_map}
	Let $f:(X,\mathcal{B}_X,{\bf b}_X)\rightarrow(Y,\mathcal{B}_Y,{\bf b}_Y)$ be a coarse proximity map. Then $f:(X,\delta_{dis,{\bf b}_X})\rightarrow(Y,\delta_{dis,{\bf b}_Y})$ is a proximity map. Moreover, if $\sigma_{1}$ is in $\mathcal{U}X,$ then the associated cluster $\sigma_{2}$ in $Y,$ as described in Proposition \ref{associated cluster}, is in $\mathcal{U}Y.$
\end{proposition}
\begin{proof}
	That $f$ is a proximity map is clear from the fact that coarse proximity maps preserve ${\bf b}$ and all sets functions preserve nontrivial intersections. Now let $\sigma_{1}$ be a cluster in $X$ that does not contain any unbounded sets. By Proposition \ref{associated cluster}, the associated cluster $\sigma_{2}$ in $Y$ is given by
	\[\sigma_{2}=\{A\subseteq Y\mid \forall C\in\sigma_{1},\,A\delta_{dis,{\bf b}_Y}f(C)\}.\]
	For contradiction, assume that $B\subseteq Y$ is a bounded set such that $B \in \sigma_2.$ Then for all $C \in \sigma_1,$ we have that $B \delta_{dis,{\bf b}_Y} f(C).$ Since $B$ is bounded and $f(C)$ is unbounded, this shows that for all $C \in \sigma_1,$ we have that $B \cap f(C) \neq \emptyset.$ Consequently, $f^{-1}(B)$ is a bounded set that intersects all $C \in \sigma_1.$ Thus, $f^{-1}(B) \in \sigma_1,$ a contradiction to $\sigma_1 \in \mathcal{U}X.$ Thus, $\sigma_2 \in \mathcal{U}Y.$

\end{proof}

\begin{corollary}\label{extension of maps to corona}
	Let $f:(X,\mathcal{B}_X,{\bf b}_X)\rightarrow(Y,\mathcal{B}_B,{\bf b}_B)$ be a coarse proximity map. Then the unique extension $f^*:\mathfrak{X}\rightarrow\mathfrak{Y}$ between Smirnov compactifications maps $\mathcal{U}X$ to $\mathcal{U}Y$. 
\end{corollary}
\begin{proof}
Follows immediately from Theorem \ref{extension_theorem} and Proposition \ref{coarse_proximity_map_induces_proximity_map}.
\end{proof}

\begin{definition}
	The map $f^*$ in Corollary \ref{extension of maps to corona} restricted in domain and codomain to $\mathcal{U}X$ and $\mathcal{U}Y$ will be denoted by $\mathcal{U}f$.
\end{definition}

\begin{proposition}\label{closeness_classes_map_to_the_same_thing}
Let $f,g:(X,\mathcal{B}_X,{\bf b}_X)\to(Y,\mathcal{B}_Y,{\bf b}_Y)$ be close coarse proximity maps. Then $\mathcal{U}f=\mathcal{U}g$.
\end{proposition}
\begin{proof}
Let $\sigma_1$ be an element of $\mathcal{U}X.$ Let $\sigma_2$ and $\sigma_2'$ be the clusters in $\mathcal{U}Y$ corresponding to the images of $\sigma_1$ under $f$ and $g,$ respectively. Let $A \in \sigma_2.$ Then for all $C \in \sigma_1,$ we have that $A \delta_{dis, {\bf b}_Y} f(C).$ Since $f(C) \phi_{{\bf b}_X} g(C)$ (where $\phi_{{\bf b}_X}$ is the weak asymptotic resemblance induced by ${\bf b}_X$), this shows that $A \delta_{dis, {\bf b}_Y} g(C).$ Since $C$ was an arbitrary element of $\sigma_1,$ this implies that $A \in \sigma_2.$ Thus, $\sigma_2 \subseteq \sigma_2'.$ The opposite inclusion follows similarly.
\end{proof}

\begin{proposition}\label{functor_propoerties}
Let $f:(X,\mathcal{B}_X,{\bf b}_X)\to(Y,\mathcal{B}_Y,{\bf b}_Y)$ and $g:(Y,\mathcal{B}_Y,{\bf b}_Y)\to(Z,\mathcal{B}_Z,{\bf b}_Z)$ be coarse proximity maps. 
 Then the following are true:
\begin{enumerate}
\item $\mathcal{U}id_X = id_{\mathcal{U}X},$
\item $\mathcal{U}(g \circ f) = \mathcal{U}g \circ \mathcal{U}f$.
\end{enumerate}
\end{proposition}
\begin{proof}
Recall that given any proximity space and any clusters $\sigma$ and $\sigma'$ in that proximity space, then $\sigma \subseteq \sigma'$ implies $\sigma = \sigma'.$ Keeping that in mind, $(1)$ is immediate. To see $(2),$ let $\sigma_1$ be an element of $\mathcal{U}X,$ $\sigma_2$ the associated cluster in $\mathcal{U}Y$ (through $f$), and $\sigma_3$ the associated cluster in $\mathcal{U}Z$ (through $g$), i.e.,
	\[\sigma_{2}=\{B\subseteq Y\mid \forall A\in\sigma_{1},\,B\delta_{dis,{\bf b}_Y}f(A)\},\]
	\[\sigma_{3}=\{C\subseteq Z\mid \forall B\in\sigma_{2},\,C\delta_{dis,{\bf b}_Z}g(B)\}.\]
 Let $\sigma_3'$ be the associated cluster in $\mathcal{U}Z$ (through $g \circ f),$ i.e., 
 	\[\sigma_3'=\{C\subseteq Z\mid \forall A\in\sigma_{1},\,C\delta_{dis,{\bf b}_Z}(g \circ f)(A)\}.\]
 
  Let $C$ be any element of $\sigma_3.$ Then $C\delta_{dis,{\bf b}_Z}g(B)$ for any $B$ in $\sigma_2.$ Since $f(A)$ is in $\sigma_2$ for any $A$ in $\sigma_1,$ this in particular shows that $C\delta_{dis,{\bf b}_Z}g(f(A))$ for any $A$ in $\sigma_1.$ Thus, $C$ is in $\sigma_3'.$ Consequently, $\sigma_3 \subseteq \sigma_3',$ and thus $\sigma_3 = \sigma_3'.$
  \end{proof}

\begin{theorem}\label{functor_theorem}
	The assignment of the compact separated proximity space $\mathcal{U}X$ to the coarse proximity space $(X,\mathcal{B},{\bf b})$ together with the assignment of the proximity map $\mathcal{U}f$ to a closeness class of coarse proximity maps $[f]:(X,\mathcal{B}_X,{\bf b}_X)\rightarrow(Y,\mathcal{B}_Y,{\bf b}_Y)$ makes up a functor from the category {\bf CrsProx} of coarse proximity spaces with closeness classes of coarse proximity maps to the category {\bf CSProx} of compact separated proximity spaces with proximity maps.
\end{theorem}
\begin{proof}
This is immediate from Corollary \ref{extension of maps to corona}, Proposition \ref{closeness_classes_map_to_the_same_thing}, and Proposition \ref{functor_propoerties}.
\end{proof}

\begin{corollary}\label{invariance_of_boundaries_under_proximal_coarse_equivalence}
If $f:(X,\mathcal{B}_X,{\bf b}_X) \to (Y,\mathcal{B}_Y,{\bf b}_Y)$ is a is a coarse proximity isomorphism, then $Uf$ is a proximity isomorphism. In particular, if $(X,\mathcal{B}_X,{\bf b}_X)$ and $(Y,\mathcal{B}_Y,{\bf b}_Y)$ are coarse proximity isomorphic, then $\mathcal{U}X$ and $\mathcal{U}Y$ are homeomorphic.
\end{corollary}
\begin{proof}
Immediate from Theorem \ref{functor_theorem}.
\end{proof}

\section{Coarse Proximities from Compactifications}\label{compact hausdorff spaces as boundaries of coarse proximity spaces}

In this section, we will show that every compactification $\overline{X}$ of a locally compact Hausdorff space $X$ induces a coarse proximity structure on $X$ whose boundary is homeomorphic to $\overline{X}\setminus X$. Keeping with the convention introduced in the previous section, for any $A \subseteq X$, $tr(A)$ will denote $tr_{\overline{X}, \overline{X}\setminus X}(A),$ as defined in the introduction.

We begin by fixing a locally compact space $X$ with compactification $\overline{X}$. Recall that $\mathcal{B}_C$ denotes the bornology of precompact sets. We can then define a coarse proximity on $(X,\mathcal{B}_{C})$ using the compactification $\overline{X}$.

\begin{proposition}\label{coarse proximity induced by compactification}
	The relation ${\bf b}$ on subsets of $X$ defined by
	\[A{\bf b}C\iff tr(A)\cap tr(C)\neq\emptyset\]
is a coarse proximity on $(X,\mathcal{B}_{C})$. We will call the coarse proximity space $(X,\mathcal{B}_{C},{\bf b})$ the \textbf{coarse proximity structure induced on $X$ by the compactification $\overline{X}$}.
\end{proposition}

\begin{proof}
	 We are only going to verify the strong axiom as the other axioms are straightforward. Let $A,C\subseteq X$ be such that $A\bar{\bf b}C$. This means that $tr(A)$ and $tr(C)$ are disjoint and closed subsets of $\overline{X}\setminus X.$ In particular, they are disjoint and closed subsets of $\overline{X}$. Since $\overline{X}$ is normal and Hausdorff, there exist open sets $U_1$ and $U_2$ of $\overline{X}$ such that $tr(A) \subseteq U_1, tr(C) \subseteq U_2$ and $cl_{\overline{X}}(U_1)$ and $cl_{\overline{X}}(U_2)$ are disjoint. Define $E=U_1 \cap X$. Notice that $E$ is nonempty, since $U_1$ is open in $\overline{X}$ and $X$ is dense in $\overline{X}$. We claim that $A\bar{\bf b}(X\setminus E)$ and $C\bar{\bf b}E$. The second of these follows from 
	\[tr(E) \subseteq tr(U_{1}) \subseteq cl_{\overline{X}}(U_1) \text{ and } tr(C)\cap cl_{\overline{X}}(U_{1})=\emptyset.\]
	Likewise, $A\bar{\bf b}(X\setminus E)$ follows from
	\[tr(X\setminus E)=tr(X\setminus U_{1}) \subseteq cl_{\overline{X}}(\overline{X} \setminus U_1)=\overline{X}\setminus U_1\]
	and 
	\[ tr(A) \cap (\overline{X} \setminus U_1) = \emptyset.\]
	This establishes the strong axiom for ${\bf b}$.
\end{proof}

\begin{theorem}\label{general_theorem}
	The boundary $\mathcal{U}X$ of the coarse proximity space $(X,\mathcal{B}_{C},{\bf b})$, where ${\bf b}$ is defined as in Proposition \ref{coarse proximity induced by compactification}, is homeomorphic to $\overline{X}\setminus X$. 
\end{theorem}
\begin{proof}
	Let $\delta$ be the subspace proximity on $\overline{X}\setminus X$, as in Example \ref{unique_proximity}. Let $\delta_{dis, {\bf b}}^{*}$ be the Smirnov proximity on the set of all clusters in $X$ induced by the discrete extension of ${\bf b}$. We will show the result via three claims. 

\textbf{Claim 1:} $_{x}\sigma=\{A\subseteq X\mid x\in tr(A)\}$ is a cluster of $X$ for all $x\in\overline{X}\setminus X$.
\begin{proof}[Proof of Claim 1]
	Notice that the only axiom that is not clear is that if $C\delta_{dis, {\bf b}}^{*}A$ for all $A\in {}_{x}\sigma,$ then $C\in {}_{x}\sigma$. However, if $C\notin {}_{x}\sigma,$ then by definition $cl_{\overline{X}}(C)\cap \{x\}=\emptyset$. In particular, $x$ is disjoint from $cl_{\overline{X}}(C).$ Since $\overline{X}$ is compact Hausdorff, we can find an open set $U$ in $\overline{X}$ such that $x \in U$ and $cl_{\overline{X}}(U)$ is disjoint from $cl_{\overline{X}}(C).$ Consider $U \cap X.$ Since $U$ is open and $X$ is dense in $\overline{X},$ $U \cap X$ is a nonempty subset of $X$. Also, $U \cap X$ is in ${}_x\sigma.$ Finally, since $cl_{\overline{X}}(U)$ is disjoint from $cl_{\overline{X}}(C),$ we know that $C \bar{\delta}_{dis,{\bf b}}(U \cap X),$ a contradiction to $C\delta_{dis,{\bf b}}A$ for all $A\in {}_{x}\sigma.$ 
\end{proof}

\textbf{Claim 2:} $\tilde{X}=\{_{x}\sigma\mid x\in\overline{X}\setminus X\}$, (equipped with the proximity $\delta_{dis,{\bf b}}^{*}$) is homeomorphic to $\overline{X}\setminus X$.
\begin{proof}[Proof of Claim 2]
	Define $f: \overline{X}\setminus X \to \tilde{X}$ by $f(x)={}_{x}\sigma.$ It is clear that $f$ is surjective. To show that $f$ is injective, let $x$ and $y$ be elements of $\overline{X}\setminus X$ such that $x \neq y.$ Consequently, there exist two open sets $U_1$ and $U_2$ of $\overline{X}$ such that $x \in U_1, y \in U_2$ and $cl_{\overline{X}}(U_1)$ is disjoint from $cl_{\overline{X}}(U_2).$ Thus, $U_1 \cap X$ and $U_2 \cap X$ are nonempty subsets of $X$ whose closures in $\overline{X}$ do not intersect and $(U_1 \cap X)\in {}_x\sigma$ and $(U_2 \cap X)\in {}_y\sigma.$ In particular, since the closures of $U_1 \cap X$ and $U_2 \cap X$ in $\overline{X}$  intersect, they cannot be in the same cluster. Thus, ${}_x\sigma \neq {}_y\sigma.$ To see that $f$ is a proximity map, let $A$ and $B$ subsets of $\overline{X} \setminus X$ such that $A \delta B,$ i.e., $cl_{\overline{X}\setminus X}(A) \cap cl_{\overline{X}\setminus X}(B) \neq \emptyset.$ Notice that
	\[f(A)=\{{}_{x}\sigma\mid x \in A\},\]
	\[f(B)=\{{}_{x}\sigma\mid x \in B\}.\]
	Let $C$ and $D$ be subsets of $X$ that absorb $f(A)$ and $f(B),$ respectively. We claim that this implies that $cl_{\overline{X}\setminus X}(A)\subseteq tr(C)$ and $cl_{\overline{X}\setminus X}(B)\subseteq tr(D)$. To see this, note that $C$ absorbing $f(A)$ implies that $C\in {}_{a}\sigma$ for all $a\in A,$ which by definition gives us that $a\in tr(C)$ for all $a\in A$. Because $A\subseteq \overline{X}\setminus X$, this gives us that $A\subseteq tr(C)$. Since  $tr(C)$ is closed in $\overline{X}\setminus X,$ we have that $cl_{\overline{X}\setminus X}(A)\subseteq tr(C)$. A similar argument shows that $cl_{\overline{X}\setminus X}(B)\subseteq tr(D)$. Because $cl_{\overline{X}\setminus X}(A)\subseteq tr(C)$, $cl_{\overline{X}\setminus X}(B)\subseteq tr(D)$, and $cl_{\overline{X}\setminus X}(A)\cap cl_{\overline{X}\setminus X}(B)\neq\emptyset$ we have that $tr(C)\cap tr(D)\neq\emptyset.$ This implies that $C{\bf b}D,$ and consequently we have that $C\delta_{dis,{\bf b}}D.$  Since $C$ and $D$ were arbitrary absorbing sets, this shows that $f(A) \delta_1^* f(B).$ Thus, $f$ is a bijective proximity map from a compact space to a Hausdorff space, which shows that $f$ is a homeomorphism.
\end{proof}

\textbf{Claim 3:} $\mathcal{U}X=\tilde{X}$
\begin{proof}[Proof of Claim 3]
	Notice that for any $x \in \overline{X} \setminus X,$ we have that ${}_x\sigma$ contains only unbounded sets (for if ${}_x\sigma$ contains a bounded set $A,$ then $tr(A)$ is empty, and thus cannot contain $x$). Thus, $\tilde{X} \subseteq \mathcal{U}X.$ To see the opposite inclusion, we can show that $\tilde{X}$ is dense in $\mathcal{U}X.$ Showing that $\tilde{X}$ is dense in $\mathcal{U}X$ is sufficient, since $\tilde{X}$ is homeomorphic to $\overline{X}\setminus X,$ and consequently it is compact (and thus closed) in $\mathcal{U}X.$ Being a dense closed subset of $\mathcal{U}X,$ $\tilde{X}$ will have to equal $\mathcal{U}X.$ To show that $\tilde{X}$ is dense in $\mathcal{U}X,$ let $\sigma\in\mathcal{U}X$ and assume that $\{\sigma\}\bar{\delta}_{dis,{\bf b}}^{*}\tilde{X}$. Then there are absorbing sets $C$ for $\sigma$ and $A$ for $\tilde{X}$ such that $tr(C)\cap tr(A)=\emptyset$. Because $tr(A)$ must equal $\overline{X}\setminus X$ (since $A$ is in ${}_x\sigma$ for all $x \in \overline{X}\setminus X$), this implies that $tr(C)=\emptyset$, which implies that $C$ is bounded. But this cannot be, since $C$ is in $\sigma$ and $\sigma$ by definition contains only unbounded sets. Thus, by contradiction it has to be that $\{\sigma\}\delta_{dis,{\bf b}}^{*}\tilde{X}.$ Since $\sigma$ was an arbitrary element of $\mathcal{U}X,$ this shows that $\tilde{X}$ is dense in $\mathcal{U}X$, which shows the claim and finishes the proof of the theorem.
\end{proof} 
\end{proof}

An immediate consequence of the preceding results of this section is the following:

\begin{corollary}
	Every compact Hausdorff space arises as the boundary of a coarse proximity space.
\end{corollary}
\begin{proof}
If $X$ is a given compact Hausdorff space, then we may take $X\times[0,1]$ to be a compactification of $Y=X\times [0,1)$. Then the coarse proximity structure on $Y$ induced by $X\times[0,1]$ is such that $\mathcal{U}Y$ is homeomorphic to $X$.
\end{proof}

In the next four sections, we are going to show how some boundaries of well-known compactifications can be realized as boundaries of coarse proximity spaces. In the next two sections, we use Theorem \ref{general_theorem} to show how the Gromov boundary and the visual boundary arise as boundaries of coarse proximity structures.
In the final two sections, we describe coarse proximities whose boundaries are homeomorphic to the Higson corona and the Freudenthal boundary without using Theorem \ref{general_theorem}.

\section{The Gromov Boundary}\label{the gromov boundary}

In this section, we will briefly review the construction of the Gromov boundary of hyperbolic metric spaces. The results and definitions outlining the construction are as they appear in \cite{Drutu} and \cite{Kapovich}. As the Gromov boundary of a hyperbolic metric space $X$ compactifies $X,$ we may treat it as the boundary of a compactification. This allows us (by using known characterizations of the Gromov boundary and Theorem \ref{general_theorem}) to describe the coarse proximity structure on a hyperbolic metric space whose corresponding boundary is homeomorphic to the Gromov boundary of that space (see Theorem \ref{Gromov Coarse Proximity Structure}).

\begin{definition}
	Let $(X,d)$ be a metric space and $x,y,p\in X$. The {\bf Gromov product} of $x$ and $y$ with respect to $p$ is
	\[(x,y)_{p}=\frac{1}{2}(d(x,p)+d(y,p)-d(x,y)).\]
\end{definition}

\begin{definition}\label{general hyperbolicity}
	Let $(X,d)$ be a metric space. Then $X$ is said to be \textbf{$\delta$-hyperbolic} for some real number $\delta<\infty$ if for all $x,y,z,p\in X,$
	
	\[(x,y)_{p}\geq\min\{(x,z)_{p},(y,z)_{p}\}-\delta.\]
\end{definition}

We note that this definition of hyperbolicity for a metric space is compatible with an alternative characterization of hyperbolicity within geodesic metric spaces due to Rips (see \cite{Roe} or \cite{Nowak}).

For the remainder of this section, let $(X,d)$ be an arbitrary $\delta$-hyperbolic metric space with a fixed based point $p.$

\begin{definition}
	A sequence $(x_{n})$ in $X$ is said to {\bf converge at infinity} if
	\[\lim_{(m,n)\to\infty}(x_{m},x_{n})_{p}=\infty.\]
\end{definition}
\begin{definition}
	Two sequences $(x_{n})$ and $(y_{n})$ in $X$ converging at infinity are said to be \textbf{equivalent}, denoted $(x_{n})\sim(y_{n})$, if
	\[\lim_{n\to\infty}(x_{n},y_{n})_{p}=\infty.\]
\end{definition}

The relation $\sim$ is an equivalence relation on sequences in $X$ that converge at infinity. We denote the equivalence class of a sequence $(x_{n})$ in $X$ converging at infinity by $[(x_n)]$ and  the set of all equivalence classes of such sequences in $X$ by $\partial X.$ We will proceed to define a topology (as in \cite{Drutu} and \cite{Kapovich}) on $\partial X$ and $\overline{X}=X\cup\partial X$ that makes both $\partial X$ and $\overline{X}$ into compact Hausdorff spaces.

To do that, identify the points of $X$ with the set of sequences in $X$ that converge to $x$. Then, extend the Gromov product on $X$ to $\overline{X}$ in the following way: 
for $\eta=[(x_{n})]\in\partial X$ and $\xi=[(y_{n})]\in\partial X,$ define:
\[(\eta,\xi)_{p}=\inf\liminf_{(m,n)\to\infty}(x_{m},y_{n})_{p},\]
where the infimum is taken over all representative sequences. If $y\in X,$ then define
\[(\eta,y)_p=\inf\liminf_{n\to\infty}(x_{n},y)_{p}\]
These products can be written most generally for $x,y\in \overline{X}$ by writing
\[(x,y)_{p}=\inf\liminf_{i\to\infty}(x_{i},y_{i})_{p},\]
where the infimum is taken over all representative sequences for $x$ and $y$.

The topology on $\overline{X}$, called the \textbf{Gromov topology}, is given by equipping $X$ with its metric topology and defining a neighborhood basis at each $\eta\in\partial X$ by defining the sets
\[U_{\eta,R}=\{x\in X\cup\partial X\mid(\eta,x)_{p}>R\}.\]
The topology on $\partial X$ and $\overline{X}$ are such that $\overline{X}$ and $\partial X$ are both compact Hausdorff spaces (a more detailed presentation of this construction can be found in \cite{Drutu} or \cite{Kapovich}). Since we identified the points of $X$ with the set of sequences in $X$ that converge to $x,$ we can think of $X$ as a dense subset of $\overline{X}.$ Consequently, we will call $\overline{X}$ the \textbf{Gromov compactification} of $X$ and $\partial X$ the \textbf{Gromov boundary}. 
By using the definition of the topology on $\overline{X}$, one can show that a sequence $(x_{n})$ in $\overline{X}$ converges to some $\eta\in\partial X$ if and only if
\[\lim_{n\to\infty}(x_{n},\eta)_{p}=\infty.\]
In particular, we have the following:

\begin{proposition}\label{convergence criterion for gromov boundary}
	Given a $\delta$-hyperbolic metric space $X$ with the Gromov boundary $\partial X$, a sequence $(x_{n})$ in $X$ converges to a point $\eta\in\partial X$ if and only if $(x_{n})$ converges at infinity and $[(x_{n})]=\eta.$ \hfill $\square$
\end{proposition}
\begin{proof}
See Lemma 11.101 in \cite{Drutu}.
\end{proof}

	Note that $\overline{X}$ equipped with the Gromov topology is first countable. Thus, if $A\subseteq X$, then we have that $\eta\in tr(A)=tr_{\overline{X}, \partial X}(A)$
if and only if there is a sequence $(x_{n})$ in $A$ that converges at infinity and $[(x_{n})]=\eta$.

The above characterization of the intersection of the closure of $A$ in $\overline{X}$ with the Gromov boundary suggests the definition of a coarse proximity structure.

\begin{theorem}\label{Gromov Coarse Proximity Structure}
Let $(X,d)$ be a proper $\delta$-hyperbolic metric space with the Gromov boundary $\partial X$. For any two sets $A,B\subseteq X,$ define $A{\bf b}_GB$ if and only if there are sequences $(x_{n})$ in $A$ and $(y_{n})$ in $B$ that converge at infinity and are equivalent via the relation $\sim$. Then the triple $(X,\mathcal{B}_d,{\bf b}_G)$ is a coarse proximity space whose boundary $\mathcal{U}_{{\bf b}_G}X$ is homeomorphic to the Gromov boundary $\partial X$.
\end{theorem}
\begin{proof}
In this case, the bornology $\mathcal{B}_{d}$ of metrically bounded subsets is identical to the bornology $\mathcal{B}_{C}$ of subsets of $X$ whose closures in $X$ are compact. Proposition \ref{convergence criterion for gromov boundary} tells us that if $A\subseteq X$ is unbounded, then a point $\eta\in\partial X$ is in $tr(A)$ if and only if there is a sequence $(x_{n})$ in $A$ that converges at infinity and whose equivalence class is $\eta$. Said differently, $\eta\in tr(A)$ if and only if there is a sequence in $A$ that converges to $\eta$ in the topology on $\overline{X}$ described above. Then for subsets $A,C\subseteq X$ we have that $A{\bf b}_{G}C$ if and only if $tr(A)\cap tr(C)\neq\emptyset$. That relation ${\bf b}_{G}$ is then precisely the coarse proximity structure on $(X,\mathcal{B}_{C})$ induced by the compactification $\overline{X}$. Theorem \ref{general_theorem} then gives us that $\mathcal{U}_{{\bf b}_{G}}X$ is homeomorphic to $\partial X$, the Gromov boundary of $X$. 
\end{proof}

\begin{definition}
	For a proper $\delta$-hyperbolic metric space $X$, the coarse proximity structure $(X,\mathcal{B}_d,{\bf b}_G)$ as described in Theorem \ref{Gromov Coarse Proximity Structure} will be called the {\bf Gromov coarse proximity structure} on $X$, and ${\bf b}_G$ will be called the \textbf{Gromov coarse proximity}.
\end{definition}

\section{The Visual Boundary}\label{The Visual Boundary}

Our next example of a boundary of a coarse proximity space will be the visual boundary assigned to a proper and complete Cat$(0)$ metric space. Our construction of the visual boundary and the visual compactification is as can be found in \cite{curtis}. An alternative construction can be found in \cite{BridsonHaefliger}.  
The main theorem of this section is Theorem \ref{visual coarse proximity}, which describes the coarse proximity on a proper and complete Cat(0) space which induces a boundary homeomorphic to the visual boundary.

\begin{definition}
	A {\bf geodesic ray} in a metric space $(X,d)$ is a map $\gamma:[0,\infty)\to X$ such that $d(\gamma(t),\gamma(t^{\prime}))=|t-t^{\prime}|$ for all $t\in[0,\infty)$. The geodesic ray $\gamma$ is said to be \textbf{based at} $x_{0}$ if $\gamma(0)=x_{0}$. A \textbf{geodesic segment} $\mu$ in $X$ is an isometric embedding of an interval $[0,t]$ into $X$. The geodesic segment $\mu:[0,t] \to X$ is said to \textbf{be between} $x\in X$ and $y \in X$ if $\mu(0)=x$ and $\mu(t)=y$.  A metric space $X$ is called a \textbf{geodesic space} if for every $x,y\in X$ there is a geodesic segment $\mu:[0,d(x,y)]\rightarrow X$ between $x$ and $y$. 
\end{definition}

\begin{definition}
	Given a geodesic metric space $(X,d)$, a {\bf geodesic triangle} $\Delta$ in $X$ consists of three vertices $a,b,c\in X$ and images of three geodesic segments between the three possible pairs of distinct vertices.
\end{definition}

\begin{definition}
	Let $(X,d)$ be a geodesic metric space and let $\Delta$ be a geodesic triangle in $X$ with vertices $a,b,$ and $c$. A {\bf comparison triangle} for $\Delta$ is a geodesic triangle $\Delta'$  in $\mathbb{R}^2$ (with the usual Euclidean metric) with vertices $a', b',$ and $c'$ such that $d(a,b)=\lVert  a'-b'\rVert,\,d(b,c)=\lVert b'-c'\rVert,$ and $d(a,c)=\lVert a'-c'\rVert$. 
\end{definition}

\begin{definition}
Let $(X,d)$ be a geodesic metric space and $\Delta$ a geodesic triangle in $X$ with vertices $a,b,$ and $c$. Let $\Delta'$ be a comparison triangle for $\Delta$ in $\mathbb{R}^2$ with vertices $a',b',$ and $c'$. Let $x \in \Delta$ belong to the image of the geodesic segment between $a$ and $b$. Then by $\bar{x}$ we mean the unique point in $\mathbb{R}^2$ such that $d(a,x)=\lVert a'-\bar{x}\rVert$ and $d(x,b)=\lVert \bar{x}-b'\rVert$.
\end{definition}

Notice that it is possible for $x \in \Delta$ to belong to more than one geodesic segment in $\Delta$, and thus the choice of $\bar{x}$ depends on the choice of the geodesic segment to which $x$ belongs.

\begin{definition}
	Let $(X,d)$ be a geodesic metric space. We say that $X$ is a {\bf Cat(0) space} if given a geodesic triangle $\Delta\subseteq X$ with comparison triangle $\Delta'\subseteq\mathbb{R}^{2},$ we have that for all $x,y$ in $\Delta,$ $d(x,y)\leq \lVert\bar{x}-\bar{y}\rVert$ for all possible choices of $\bar{x}$ and $\bar{y}.$
\end{definition}

The boundary of a proper and complete Cat$(0)$ space $X$ is going to be the set $\partial X$ of geodesic rays based at some fixed $x_{0}\in X$ equipped with a compact and Hausdorff topology. We will describe a topology on $\overline{X}=X\cup\partial X$ that makes $\overline{X}$ into a compact Hausdorff space and $\partial X$ a closed subset thereof. For distinct points $x,y\in X$, we will denote the unique image of a geodesic segment between them as $[x,y]$ (the uniqueness is proven in \cite{BridsonHaefliger}).
\vspace{\baselineskip}

Fix a basepoint $x_{0}\in X$. Let $\partial X$ be the set of geodesic rays in $X$ based at $x_{0}$. For each $\gamma\in\partial X$ and $R,\epsilon>0$, define
\[V(\gamma,R,\epsilon)=\{\alpha\in\partial X\mid d(\gamma(R),\alpha([0, \infty)))<\epsilon\},\]
\[C(\gamma,R,\epsilon)=\{x\in X\mid d(\gamma(R),[x_{0},x])<\epsilon\},\]
\[D(\gamma,R,\epsilon)=V(\gamma,R,\epsilon)\cup C(\gamma,R,\epsilon).\]

\noindent
Let $\mathcal{W}$ be the collection of sets of the form $D(\gamma,R,\epsilon)$ together with the open metric balls of $X$. It is then straightforward to show that $\mathcal{W}$ is a basis of a compact and Hausdorff topology on $\overline{X}$, and that $\partial X$ is a closed and dense subset of $\overline{X}$ in that topology. The compactification $\overline X$ is called the \textbf{visual compactification} of $X$ and the boundary $\partial X$ is called the \textbf{visual boundary} of the visual compactification $\overline{X}$.

With this setup in mind, we can describe a coarse proximity structure on a proper and complete Cat$(0)$ space $X$ whose corresponding boundary is the visual boundary of the visual compactification of $X$.

\begin{theorem}\label{visual coarse proximity}
	Let $(X,d)$ be a proper and complete Cat$(0)$ space with fixed point $x_{0}$. 
Define a relation ${\bf b}_{V,x_{0}}$ on $(X,\mathcal{B}_{C})$ by defining for all $A,B \subseteq X:$
	\begin{equation*}
\begin{split}
A{\bf b}_{V,x_{0}}B\iff & \exists(x_{n})\subseteq A,\,(y_{n})\subseteq B,\text{ and } \eta\in\partial X\text{ such that } \lim x_n=\lim y_n = \eta,\\
\end{split}
\end{equation*}
	where the above limits are taken in the visual compactification of $X$ constructed using the basepoint $x_{0}$. The triple $(X,\mathcal{B}_C,{\bf b}_{V,x_{0}})$ is a coarse proximity space whose boundary $\mathcal{U}_{{\bf b}_{V,x_{0}}}X$ is homeomorphic to the visual boundary $\partial X$.
\end{theorem} 
\begin{proof}
As the compactification $\overline{X}$ is first countable, we have that a point $\eta\in\partial X$ is in $cl_{\overline{X}}(A)$ for some $A\subseteq X$ if and only if there is a sequence $(x_{n})$ in $A$ that converges to $\eta$ in the visual compactification of $X$. Then, for subsets $A,B\subseteq X$ we have that $A{\bf b}_{V,x_{0}} B$ if and only if $tr_{\overline{X}, \partial X}(A)\cap tr_{\overline{X}, \partial X}(B)\neq\emptyset$. Then, $(X,\mathcal{B}_{C},{\bf b}_{V,x_{0}})$ is precisely the coarse proximity structure induced by the compactification $\overline{X}$. Consequently, by Theorem \ref{general_theorem} the boundary $\mathcal{U}_{{\bf b}_{V,x_{0}}}X$ is homeomorphic to $\partial X$, the visual boundary of $X$.
\end{proof}

\begin{definition}
For a proper and complete Cat$(0)$ space $(X,d)$, the coarse proximity structure $(X,\mathcal{B}_C,{\bf b}_{V,x_0})$ will be called the {\bf visual coarse proximity} on the proper and complete Cat$(0)$ space $X$, and ${\bf b}_{V,x_0}$ will be called the \textbf{visual coarse proximity}.
\end{definition}

We conclude this section with the remark that the choice of the basepoint does not change the topology on the visual boundary. For the sake of brevity we have not included an argument to that effect here. Interested readers may see \cite{BridsonHaefliger} for details on how to construct the boundary in a basepoint independent way. 

\section{The Higson Corona}\label{the higson boundary}

In this section, we show how the Higson corona of a proper metric space is a boundary of a particular coarse proximity space (whose underlying base space is that proper metric space).

Recall that $(X,d)$ is called \textbf{$k$-discrete} for some $k>0$, if, for any distinct $x,y \in X,$ we have $d(x,y) \geq k.$

\begin{theorem}\label{higson proximity}
	Let $(X,d)$ be a proper metric space, and let $\delta_H$ be a binary relation on the power set of $X$ defined by: 
	\[A\delta_H B \quad \Longleftrightarrow \quad d(A,B)=0  \text{ or } A{\bf b}_d B,\]
where $d(A,B)=\inf\{ d(x,y) \mid x \in A, y \in B \}.$ Then $\delta_H$ is a separated proximity that is compatible with the topology on $X.$ We will call this proximity the {\bf Higson proximity} associated to the metric space $(X,d).$
\end{theorem}
\begin{proof}
The proof is straightforward.
\end{proof}

In other words, two subsets $A$ and $B$ of a proper metric space $X$ are close via the Higson proximity if the distance between them is $0$ or if they are metrically coarsely close. 
If $(X,d)$ is a $k$-discrete proper metric space for some $k>0,$ then the condition $d(A,B)=0$ can be replaced with $A\cap B\neq\emptyset$. Thus, when the $k$-discrete proper metric space is equipped with the metric coarse proximity ${\bf b}_d,$ the Higson proximity is the discrete extension of ${\bf b}_d,$ i.e.,
\[\delta_H=\delta_{dis, {\bf b}_d}.\]

\begin{definition}
	Let $(X,d)$ be a proper metric space. The {\bf Higson compactification} of $X$, denoted $hX$, is the Smirnov compactification of $(X,\delta_H)$. The {\bf Higson corona} (i.e., \textbf{Higson boundary}) is the corresponding Smirnov boundary $hX \setminus X$ and is customarily denoted by $\nu X$.
\end{definition}

The equivalence of the above definition with the standard way of constructing the Higson compactification (see \cite{Roe} or \cite{connectedcoronas} for the standard construction) has been given in \cite{Dranishnikovasymptotictopology}.

\begin{proposition}\label{Higson_corona_only_contains_unbounded_clusters}
	Let $(X,d)$ be a proper $k$-discrete metric space for some $k>0$. Then a cluster $\sigma\in hX$ is an element of the Higson corona $\nu X$ if and only if $\sigma$ does not contain any bounded sets.
\end{proposition}
\begin{proof}
Bounded sets in $X$ are necessarily compact (actually finite) by the properness and the $k$-discreteness of $X$. Consequently, the forward direction follows from Proposition \ref{elements of Smirnov boundary don't contain compact sets}. Conversely, if $\sigma\in hX$ does not contain any bounded sets, then clearly $\sigma \neq \sigma_x$ for any $x \in X.$ Thus, $\sigma \in hX \setminus X=\nu X.$
\end{proof}

\begin{theorem}
If $(X,d)$ is a proper metric space, then $\mathcal{U}_{{\bf b}_d}X$  is homeomorphic to the Higson corona $\nu X.$
\end{theorem}
\begin{proof}
If $(X,d)$ is a proper metric space, then it is well-known that $X$ is coarsely equivalent to a $1$-discrete proper subset $X^{\prime}$ of itself. Because 
\[\delta_H=\delta_{dis, {\bf b}_d},\]
Proposition \ref{Higson_corona_only_contains_unbounded_clusters} tells us that $\mathcal{U}_{{\bf b}_d}X^{\prime}$ is the Higson corona $\nu X'$ of $X^{\prime}.$ We also know that
\begin{itemize}
\item  coarsely equivalent spaces have homeomorphic Higson coronas (see Corollary 2.42 in \cite{Roe}),
\item $\mathcal{U}_{{\bf b}_d}X$ is homeomorphic to $\mathcal{U}_{{\bf b}_d}X^{\prime}$ (coarsely equivalent metric spaces are coarse proximity isomorphic by Corollary $7.5$ in \cite{paper1}, and thus by Corollary \ref{invariance_of_boundaries_under_proximal_coarse_equivalence} their boundaries are homeomorphic).
\end{itemize}
Thus, if $\approx$ denotes a homeomorphism, we have that
\[\mathcal{U}_{{\bf b}_d}X \approx \mathcal{U}_{{\bf b}_d}X'= \nu X' \approx \nu X\]
which shows that $\mathcal{U}_{{\bf b}_d}X$ is homeomorphic to the Higson corona $\nu X.$
\end{proof}

\section{The Freudenthal Boundary}\label{the freudenthal boundary}
In this section, we explicitly construct (without using Theorem \ref{general_theorem}) a coarse proximity structure on a locally compact Hausdorff space such that the boundary of that coarse proximity space is the Freudenthal boundary. 

\begin{definition}
Let $X$ be a topological space and $A$ and $B$ two subsets of $X.$ We say that $A$ and $B$ are \textbf{separated by a compact set} if there exists a compact set $K\subseteq X$ and open sets $U_1, U_2 \subseteq X$ such that 
\begin{enumerate}
\item $X \setminus K=U_1 \cup U_2$,
\item $U_1 \cap U_2 = \emptyset,$
\item $A \subseteq U_1$ and $B\subseteq U_2.$
\end{enumerate}
\end{definition}

The following proximity characterization of the Freudenthal compactification of a locally compact Hausdorff space comes from \cite{isbell}.

\begin{definition}\label{Fcompactification}
	If $X$ is a locally compact Hausdorff space, then the {\bf Freudenthal compactification} of $X$ is the Smirnov compactification of the separated proximity $\delta_F$ on $X$, called the \textbf{Freudenthal proximity}, defined by 
	\[A\bar{\delta}_FB \quad \Longleftrightarrow \quad A \text{ and } B \text{ are separated by a compact set.}\] 
	We will denote the Freudenthal compactification of $X$ by $FX$.
\end{definition}

Recall that Proposition \ref{elements of Smirnov boundary don't contain compact sets} tells us the elements of $FX\setminus X$ do not contain any compact sets. Thus, if a bornology on $X$ consists of all precompact sets (denoted $\mathcal{B}_C$), then elements of $FX\setminus X$ do not contain any bounded sets (since clusters are closed under taking supersets).

\begin{proposition}\label{freudenthal coarse proximity}
	Let $X$ be a locally compact Hausdorff space. Define the relation ${\bf b}_F$ on the power set of $X$ by:
	\vspace{.3em}
	
	\noindent $A\bar{\bf b}_FB$ if and only if there are compact sets $D,K\subseteq X$ such that $(A\setminus D)$ and  $(B\setminus D)$ are separated by $K.$
	\vspace{.3em}
	
\noindent Then $(X,\mathcal{B}_C,{\bf b}_F)$ is a coarse proximity space.
\end{proposition}
\begin{proof}
	The only axiom that is not clear is the strong axiom. Let $A,B\subseteq X$ be such that $A\bar{\bf b}_FB$. Then by definition there are compact sets $D$ and $K$ such that there are disjoint open sets $U,V\subseteq X$ that are disjoint from $K$, cover $X\setminus K$, and contain $A\setminus D$ and $B\setminus D$, respectively. Then $V$ is a subset of $X$ such that $A\bar{\bf b}_FV$ and $(X\setminus V)\bar{\bf b}_FB$. Thus, the triple $(X,\mathcal{B}_C,{\bf b}_F)$ is a coarse proximity space.
\end{proof}

Notice that with the notation from Definition \ref{Fcompactification} and Proposition \ref{freudenthal coarse proximity}, we have that 
\[A{\bf b}_FB \quad \Longleftrightarrow \quad (A\setminus D) \delta_F (B \setminus D) \text{ for all bounded sets }D.\]

\begin{definition}
	For a locally compact Hausdorff space $X,$ the coarse proximity structure $(X,\mathcal{B}_C,{\bf b}_F)$ will be called the {\bf Freudenthal coarse proximity structure} on $X$, and ${\bf b}_F$ will be called the  {\bf Freudenthal coarse proximity}.
\end{definition}

Our goal is to show that $FX \setminus X$ is homeomorphic to $\mathcal{U}_{{\bf b}_F}X.$ The two following lemmas are used to prove Proposition \ref{freudenthal boundary is boundary of freudenthal coarse proximity structure}, which says that clusters in $FX \setminus X$ are exactly those in $\mathcal{U}_{{\bf b}_F}X.$

\begin{lemma}\label{closeness in Freudenthal clusters is realized by coarse closeness}
	Let $X$ be a locally compact Hausdorff space. If $A,B\subseteq X$ are sets such that $A,B \in \sigma$ for some $\sigma \in FX\setminus X,$ then $A{\bf b}_FB.$
\end{lemma}
\begin{proof}
	Since $A,B \in \sigma,$ Proposition \ref{elements of Smirnov boundary don't contain compact sets} implies that both $A$ and $B$ are unbounded. For contradiction, assume that $A\bar{\bf b}_FB.$ Then there is a compact set $D \subseteq X$ such that $(A\setminus D) \bar{\delta}_F (B\setminus D).$ However, by Proposition \ref{can subtract by compact sets in cluster}, $(A\setminus D)$ and $(B\setminus D)$ are elements of $\sigma,$ which in particular means that $(A\setminus D)\delta_F(B\setminus D),$ a contradiction.
\end{proof}

\begin{lemma}\label{second_axiom_lemma}
Let $X$ be a locally compact Hausdorff space. Let $A$ and $B$ be susbets of $X.$ Then
\[(A\delta_F B \text{ and } A\bar{\bf b}_FB) \quad \Longrightarrow \quad cl_X(A) \cap cl_X(B) \text{ is compact. }\]
\end{lemma}
\begin{proof}
Since $A\bar{\bf b}_FB,$ there exist compact $D$ and compact $K$ such that $X \setminus K = U_1 \cup U_2$ for some open sets $U_1$ and $U_2$ that are disjoint and $A \setminus D \subseteq U_1$ and $B \setminus D \subseteq U_2.$ If we show that $cl_X(A) \cap cl_X(B) \subseteq K\cup D,$ then $cl_X(A) \cap cl_X(B)$ is a closed subset of a compact space, and consequently is compact. To see that $cl_X(A) \cap cl_X(B) \subseteq K\cup D,$ let $x$ be an element of $cl_X(A) \cap cl_X(B).$ For contradiction, assume that $x \notin K\cup D.$ Since $x \notin K,$ without loss of generality we can assume that $x \in U_1.$ Since $X$ is regular, let $V_1$ and $V_2$ be open sets such that $x \in V_1, K \cup D \subseteq V_2,$ and $V_1 \cap V_2 = \emptyset.$ Then $x$ belongs to an open set $V_1 \cap U_1$ that is disjoint from $B$ (since $B\setminus D$ does not intersects $U_1$ and $D$ does not intersect $V_1$). This contradicts the fact that $x$ is in the closure of $B.$
\end{proof}

\begin{proposition}\label{freudenthal boundary is boundary of freudenthal coarse proximity structure}
	Let $X$ be a locally compact Hausdorff space. Let $\sigma$ be a collection of susbets of $X$. Then
\[\sigma \in FX \setminus X \Longleftrightarrow \sigma \in \mathcal{U}_{{\bf b}_F}X.\]
\end{proposition}
\begin{proof}
($\Longrightarrow$)
	Let $\sigma\in FX\setminus X$ be given. By Proposition \ref{elements of Smirnov boundary don't contain compact sets}, $\sigma$ does not contain any bounded sets. Thus, it is enough to show that $\sigma$ satisfies all the axioms of a cluster under $\delta_{dis, {\bf b}_F}$. This will show that $\sigma \in \mathcal{U}_{{\bf b}_F}X.$ To see the first axiom of a cluster, notice that by Lemma \ref{closeness in Freudenthal clusters is realized by coarse closeness}, we have that if $A,B\in\sigma$ then $A{\bf b}_FB,$ which shows that $A \delta_{dis, {\bf b}_F} B$. To prove the second axiom of a cluster, notice that 
\begin{equation*}
\begin{split}
\forall  \, B \in \sigma, A\delta_{dis, {\bf b}_F} B  & \Longleftrightarrow \forall \, B \in \sigma, A \cap B \neq  \emptyset \text{ or } A{\bf b}_F B\\
&\Longleftrightarrow \forall \, B \in \sigma, A \cap B \neq  \emptyset \text{ or } (A\setminus D) \delta_F (B\setminus D) \, \forall \, D \in \mathcal{B}_C\\
&\Longrightarrow \forall \, B \in \sigma, A \cap B \neq  \emptyset \text{ or } A \delta_F B\\
&\Longrightarrow \forall \, B \in \sigma,  A \delta_F B\\
&\Longrightarrow A \in \sigma.
\end{split}
\end{equation*}
The last axiom of a cluster is true, since $\sigma$ is also a cluster in $FX \setminus X$ by assumption.
	
($\Longleftarrow$)
Let $\sigma \in \mathcal{U}_{{\bf b}_F}X$ be given.	By definition, $\sigma$ does not contain any compact sets. Thus, if $\sigma$ is a cluster in $FX,$ then it clearly is a cluster in $FX \setminus X.$ To see that $\sigma$ is a cluster in $FX,$ we need to show the three axioms of a cluster:
\begin{enumerate}
\item if $A,B\in\sigma$ then $A\delta_F B$,
\item if $C\subseteq X$ is such that $C\delta_F A$ for all $A\in\sigma$, then $C\in\sigma$
\item if $A\cup B \in \sigma$, then $A \in \sigma$ or $B \in \sigma.$
\end{enumerate}

Axiom 3 is obvious, since $\sigma$ is also a cluster in $\mathcal{U}_{{\bf b}_F}X$ by assumption. To see axiom 1, assume  $A,B\in\sigma$. Note that if $A\bar{\delta}_FB,$ then $A$ and $B$ can be separated by a compact set. In particular, they are disjoint and  $A\bar{\bf b}_FB$. But this is a contradiction to $A\delta_{dis, {\bf b}_F}B.$ Thus, $A\delta_FB$. 

Finally, we need to prove the second axiom, which will finish the proof. Let $C\subseteq X$ be such that $C\delta_F A$ for all $A\in\sigma$. For contradiction, assume that $C \notin \sigma.$ Then there has to exist $B \in \sigma$ such that $C \bar{\delta}_{dis, {\bf b}_F} B$ (otherwise, since $\sigma$ is a cluster in $\mathcal{U}_{{\bf b}_F}X$ it would imply that $C \in \sigma$ by the second axiom of a cluster). In particular, this implies that $C \bar{\bf b}_F B.$ Since we have that $C \delta_F B$ and $C \bar{\bf b}_F B,$ Lemma \ref{second_axiom_lemma} tells us that $W=cl_X(C) \cap cl_X(B)$ is compact. Since $X$ is also locally compact, we can find a finite open cover $U_1,...,U_n$ of $W$ such that each element in that cover has compact closure.  Define
	\[W'=\bigcup_{i=1}^n cl_X(U_i).\]
It is clear that $W \subseteq W'$ and $W'$ is compact. Since $B\in \sigma$ and $W'$ is compact, $B\setminus W' \in \sigma.$ Consequently, by the assumption about $C$ we have $C \delta_F (B \setminus W').$ Proposition \ref{cluster_containing_both} implies then that there exists a cluster $\sigma_2$ in $FX$ that contains both $C$ and $B \setminus W'.$ To show that $\sigma_2$ is not a point cluster, we prove a series of claims:\\
\textbf{Claim 1}: $cl_X(C) \cap cl_X(B \setminus W') = \emptyset.$ \\
Notice that $cl_X(C) \cap cl_X(B \setminus W') \subseteq cl_X(C) \cap cl_X(B).$ Thus, showing that $cl_X(B \setminus W')$ is disjoint from $cl_X(C) \cap cl_X(B),$ proves Claim 1. Thus, to  prove Claim 1 it is enough to prove:\\
\textbf{Claim 2}: $cl_X(B \setminus W')$ is disjoint from $cl_X(C) \cap cl_X(B).$\\
To see that claim 2 is true, notice that 
\[cl_X(C) \cap cl_X(B) =W \subseteq \text{int}(W') \subseteq X \setminus cl_X(B \setminus W'),\] 
where $\text{int}(W')$ denotes the interior of $W'$ in $X$.  This proves Claim 2 and consequently Claim 1.\\
\textbf{Claim 3}: $\sigma_2$ is not a point cluster, i.e., $\sigma_2 \in FX \setminus X.$ \\
To see that Claim 3 is true, for contradiction assume that there exists $x \in X$ such that $\{x\} \in \sigma_2.$ Under the assumptions of Claim 3, we will show that\\
\textbf{Claim 4: } neither $cl_X(C)$ nor $cl_X(B \setminus W')$ contains $x$. \\
If $x \in cl_X(C),$ then by Claim 1 we have that $cl_X(C) \cap cl_X(B \setminus W') = \emptyset.$ Thus, by regularity we can find an open set $V$ that contains $x$ and whose closure is disjoint from $cl_X(B \setminus W').$ By local compactness we can assume that the closure of $V$ is compact. In particular, the boundary of $V$ is compact. Consequently, the boundary of $V$ is a separating compact set for $x$ and $cl_X(B \setminus W').$ In other words, $\{x\} \bar{\delta}_F cl_X(B \setminus W').$ But this is a contradiction, since both $x$ and $cl_X(B \setminus W')$ are in $\sigma_2$ and thus should be $\delta_F$-close. Consequently, $x \notin cl_X(C).$ Similarly one can show that $x \notin cl_X(B \setminus W').$ Thus, $x$ is in neither $cl_X(C)$ nor $cl_X(B \setminus W').$ This proves Claim 4.\\
\textbf{Claim 3 proof cont.:} Notice that Claim $4$ implies that $x \in X \setminus (cl_X(C)\cup cl_X(B \setminus W')).$ But since $cl_X(C)\cup cl_X(B \setminus W')$ is closed, by the same argument that uses regularity we can show that this implies that $\{x\} \bar{\delta}_F (cl_X(C)\cup cl_X(B \setminus W')),$ a contradiction (since $cl_X(C)$ is in $\sigma_2,$ $cl_X(C)\cup cl_X(B \setminus W')$ is in $\sigma_2,$ and consequently $cl_X(C)\cup cl_X(B \setminus W')$ should be $\delta_F$-close to $x$). Thus, it has to be that $\sigma_2$ is not a point cluster, i.e., $\sigma_2 \in FX \setminus X.$ This shows Claim 3.

Since both $C$ and $B \setminus W'$ belong to a cluster $\sigma_2 \in FX \setminus X,$ by Proposition \ref{can subtract by compact sets in cluster} we have that both $C \setminus K$ and $(B \setminus W') \setminus K$ belong to $\sigma_2$ for any compact set $K.$ This implies that $(C \setminus K) \delta_F ((B \setminus W') \setminus K)$ for any compact $K,$ which implies that $(C \setminus K) \delta_F (B \setminus K)$ for any compact set $K.$ Consequently, $C{\bf b}_FB,$ which contradicts the original assumption about $C$. Thus, $C \in \sigma,$ finishing the proof that $\sigma$ is an element of $FX \setminus X.$
\end{proof}

\begin{theorem}
Let $X$ be a locally compact Hausdorff space. Then $\mathcal{U}_{{\bf b}_F}X$ is homeomorphic to $FX\setminus X$.
\end{theorem}
\begin{proof}
	In light of Proposition \ref{freudenthal boundary is boundary of freudenthal coarse proximity structure} it will be enough to show that the identity map $id:\mathcal{U}X\rightarrow FX\setminus X$ is continuous. As both of these spaces are compact and Hausdorff, it is enough to show that the identity map is a proximity map. Let $\mathcal{A},\mathcal{B}\subseteq \mathcal{U}_{{\bf b}_F}X$ be such that $\mathcal{A}\delta^{*}_{dis,{\bf b}_F}\mathcal{B}$ where $\delta^{*}_{dis,{\bf b}_F}$ is the proximity relation of $\mathcal{U}X$. We wish to show that $\mathcal{A}\delta^{*}_F\mathcal{B}$ where $\delta^{*}_F$ is the proximity on $FX$. Let $A,B\subseteq X$ be absorbing sets for $\mathcal{A}$ and $\mathcal{B},$ respectively. Since $\mathcal{A}\delta^{*}_{dis,{\bf b}_F}\mathcal{B},$ we know that $A \delta_{dis,{\bf b}_F}B,$ i.e., $A \cap B \neq \emptyset$ or $A{\bf b}_F B.$ Since either of these conditions imply $A\delta_FB,$ we have that $\mathcal{A}\delta^{*}_F\mathcal{B}$, which shows that $id$ is a proximity map and therefore a homeomorphism.
\end{proof}

\bibliographystyle{abbrv}
\bibliography{Boundaries_of_coarse_proximity_spaces_and_boundaries_of_compactifications}{}

\end{document}